\documentclass{amsart}[12pt]
\usepackage{epsfig,graphicx,young,amsmath, amssymb, amsfonts,fancyhdr,hyperref}
\usepackage{lineno}

\fancypagestyle{titlepage}{
\fancyhead[R]{}
\fancyhead[CL]{}
\cfoot{\thepage}
\footskip=20pt
}

\newtheoremstyle{thm}%name style: Theorems, Lemas, Corollaries
  {12pt}		  % space above
  {0pt}  % space below
  {\sl}  % bofy font
  {\parindent}     % ident - empty=no indent,  \parindent= paragraph indent
  {\bf}  % thm head font
  {. }    % punctuation after thm head
  { }    % space after thm head: `` ``=normal \newline=linebreak
  {}     % thm head specification
\theoremstyle{thm}
\newtheorem{thm}{thm}[section]  % 1st argument is your name for it
\newtheorem{lemma}[thm]{Lemma}     % 2nd argument is what is printed
\newtheorem{corollary}[thm]{Corollary}
\newtheorem{prop}[thm]{Proposition}
\newtheorem{conjecture}[thm]{Conjecture}
\theoremstyle{definition}
\newtheorem{example}[thm]{Example}

\numberwithin{equation}{section}

\newcommand{\vj}{\mathbf{1}}
\newcommand{\A}{\mathfrak{A}}

\newcommand{\C}{\mathbb{C}}

\newcommand{\cR}{\mathcal{R}}
\newcommand{\cL}{\mathcal{L}}
\newcommand{\cC}{\mathcal{C}}
\newcommand{\cI}{\mathcal{I}}
\newcommand{\Aut}{{\rm Aut}}
\newcommand{\Sym}{{\tt Sym}}
\newcommand{\Iso}{{\tt Iso}}
\newcommand{\id}{\mathrm{id}}
\newcommand{\trace}{\mathrm{Tr}}

\newcommand{\Section}[1]{\section{\bf #1}}

\makeatletter
\renewcommand*\@maketitle{%
  \normalfont\normalsize
  \@adminfootnotes
  \@mkboth{\@nx\shortauthors}{\@nx\shorttitle}%
  \global\topskip42\p@\relax % 5.5pc   "   "   "     "     "
  \@settitle
  \ifx\@empty\authors \else {\vskip 1em \vtop{\centering\shortauthors\@@par}} \fi
  \ifx\@empty\@date \else {\vskip 1em \vtop{\centering\@date\@@par}}\fi% MY CHANGE
  \ifx\@empty\@dedicatory
  \else
    \baselineskip18\p@
    \vtop{\centering{\footnotesize\itshape\@dedicatory\@@par}%
      \global\dimen@i\prevdepth}\prevdepth\dimen@i
  \fi
  \@setabstract
  \normalsize
  \if@titlepage
    \newpage
  \else
    \dimen@34\p@ \advance\dimen@-\baselineskip
    \vskip\dimen@\relax
  \fi
} % end \@maketitle
\renewcommand*\@adminfootnotes{%
  \let\@makefnmark\relax  \let\@thefnmark\relax
  \ifx\@empty\@subjclass\else \@footnotetext{\@setsubjclass}\fi
  \ifx\@empty\@keywords\else \@footnotetext{\@setkeywords}\fi
  \ifx\@empty\thankses\else \@footnotetext{%
    \def\par{\let\par\@par}\@setthanks}%
  \fi
}

\makeatother

\begin{document}

\title[SDP for permutation codes]{Semidefinite programming for permutation codes}

\author{Mathieu Bogaerts}
\address{\rm Mathieu Bogaerts: Facult\'e des Sciences Appliqu\'ees, 
 Universit\'e Libre de Bruxelles, Bruxelles, Belgium}
\email{mbogaert@ulb.ac.be}

\author{Peter Dukes}
\address{\rm Peter Dukes: Department of Mathematics and Statistics, 
University of Victoria, Victoria, Canada}
\email{dukes@uvic.ca}

\thanks{Research of Peter Dukes is supported by NSERC}

%\subjclass[2000]{05B05}

\date{\today}

\begin{abstract}
We initiate study of the Terwilliger algebra and related semidefinite
programming techniques for the conjugacy scheme of the symmetric group
$\Sym(n)$.  In particular, we compute orbits of ordered pairs on $\Sym(n)$
acted upon by conjugation and inversion, explore a block diagonalization of the associated
algebra, and obtain improved upper bounds on the size $M(n,d)$ of permutation codes of lengths up to
$7$.  For instance, these techniques detect the nonexistence of the
projective plane of order six via $M(6,5)<30$ and yield a new best bound
$M(7,4) \le 535$ for a challenging open case.  Each of these represents an
improvement on earlier Delsarte linear programming results.
\end{abstract}

\maketitle
\thispagestyle{titlepage}

%\hrule
%\tableofcontents
\hrule

\Section{Introduction and notation}
\label{intro}

Consider the symmetric group $\Sym(n)$ on $[n]:=\{1,\dots,n\}$.  Let $D$ be
a collection of conjugacy classes of $\Sym(n)$.  A subset $\Gamma \subseteq
\Sym(n)$ is an $(n,D)$-\emph{permutation code} if, for any two distinct
elements $\phi,\psi \in \Gamma$, their quotient $\phi \psi^{-1}$ belongs to
a class in $D$.  Note that the order of the terms is not important, since
permutations are conjugate with their inverses.

The \emph{Hamming distance} between $\phi$ and $\psi$ in $\Sym(n)$, denoted
$d_H(\phi,\psi)$, is the number of non-fixed points of $\phi \psi^{-1}$.
Equivalently, if $\phi$ and $\psi$ are written in single-line notation,
$d_H(\phi,\psi)$ counts the number of disagreements between corresponding
positions.  It follows that $d_H$ is a metric for $\Sym(n)$.  To the best of
our knowledge, this metric was first considered by Farahat in \cite{F}.

In most investigations of permutation codes, the set $D$ of admissible
conjugacy classes reflects this metric and is instead taken as some subset $D
\subseteq [n]$.  This interpretation for $D$ is that all conjugacy classes
with $n-d$ one-cycles, $d \in D$, are allowed for quotients $\phi
\psi^{-1}$.  In other words, in this context, an $(n,D)$-permutation code
$\Gamma$ is a subset of $\Sym(n)$ with all nonzero Hamming distances
belonging to $D$.  Based on the coding applications discussed below, the
standard choice for $D$ is the interval $\{d,d+1,\dots,n\}$ for some
\emph{minimum distance} $d$.

Let $M(n,D)$ denote the maximum size of an $(n,D)$-permutation code.  When
$D=\{d,d+1,\dots,n\}$, this is simply written $M(n,d)$.  Early work
determining some bounds and values for $M(n,D)$ began in the late 1970s in
\cite{DV,FD}, where the term `permutation array' was used.  Around 2000,
permutation codes enjoyed a revival of interest with the discovery in
\cite{CKL,FV} of their applications to trellis codes.  Then, the survey
article \cite{CCD} observed connections with permutation polynomials and
also initiated the first serious computational attack on lower bounds on
$M(n,d)$.  A probabilistic lower bound appears in \cite{KK}.  Meanwhile,
linear programming (LP) upper bounds were investigated, first in \cite{T},
and then subsequently in \cite{B,DS}.  A growing body of related work has
emerged, including constant composition codes, injection codes, and study of
the packing and covering radii.

The LP bound rests on Delsarte's theory of association schemes \cite{Del}
applied to the conjugacy scheme of $\Sym(n)$.  More details follow in
Section~\ref{terw}.  In various combinatorial settings (block designs and
binary codes, for instance) Delsarte LP bounds have been successfully
improved using semidefinite programming (SDP).  In general, the sizes of
matrices required for this SDP tend to grow impractically large.  So the
most successful applications of SDP to designs and codes usually begin with
an attack on the algebraic structure.  In a little more detail, block
diagonalizations of certain matrix algebras are desired in order to scale
the computations.  This is the content of Dion Gijswijt's dissertation
\cite{Gij}.

For a finite set $X$, we use $M_X(\C)$ or $\C^{X \times X}$ to denote the
algebra of $|X| \times |X|$ matrices with rows and columns indexed by $X$.
(Some canonical ordering of $X$ is usually assumed, leading to the more
standard notation $M_n(\C)$ or $\C^{n \times n}$, which we also use.)  The
$(x,y)$-entry of matrix $A \in M_X(\C)$ is here denoted $A(x,y)$.  All
linear spans and generated sub-algebras $\langle A_1,A_2, \dots \rangle$ are
over $\C$, unless otherwise noted.  The  conjugate transpose
$\overline{A^\top}$ of $A$ is denoted $A^*$ and as usual $A$ is
\emph{Hermitian} if $A^*=A$.  A sub-$\C$-algebra of $M_X(\C)$ closed under
this Hermitian conjugate is called a $\C^*$- (matrix) algebra.

\Section{Bose-Mesner and Terwilliger algebras}
\label{terw}

An $d$-\emph{class association scheme} on a finite set $X$ is a list of
binary relations $R_0,\dots,R_d$ on $X$ such that $R_0$ is the identity
relation, the relations partition $X^2$, and the following regularity
condition holds: given $x$ and $y$ with $(x,y) \in R_k$, the number of $z
\in X$ for which both $(x,z) \in R_i$ and $(z,y) \in R_j$ is a constant
depending only on $i,j,k \in \{0,\dots,d\}$.  These values $p_{ij}^k$ are
called the \emph{structure constants} or \emph{intersection numbers} of the
scheme.  Here, we are mainly interested in \emph{symmetric} schemes in which
each relation $R_i$ is symmetric and thus $p_{ij}^k=p_{ji}^k$.  Chris
Godsil's notes \cite{Godsil} offer a readily available and comprehensive
reference on association schemes.

\begin{example}
Consider $X=G$, a finite group.  The relations $R_i$ are indexed by the
conjugacy classes of $G$, where $R_0$ corresponds with the trivial conjugacy
class.  Declare $(g,h)$ in $R_i$ if and only if $gh^{-1}$ belongs to the
$i$th conjugacy class.  The fact that these $R_i$ partition $G^2$ is clear.
A character sum can compute the structure constants (see below).  This is
the \emph{conjugacy scheme} on $G$.
\end{example}

Let $|X|=N$.
Define the $N \times N$ zero-one matrices $A_0,\dots,A_d \in M_X(\C)$ by
$$A_i(x,y) =
\begin{cases}
1 & \text{if}~(x,y) \in R_i, \\
0 & \text{otherwise}.
\end{cases}
$$
We have $A_0=I$ and $A_0+\dots+A_d=J$, the all ones matrix.
And from the definition of the structure constants,
\begin{equation}
\label{structure}
A_i A_j = \sum_{k=0}^d p_{ij}^k A_k.
\end{equation}
From this it follows that $\langle A_0,\dots,A_d \rangle$ is a commutative
$\C^*$-matrix sub-algebra of $M_X(\C)$ of dimension $d+1$.
This is the \emph{Bose-Mesner algebra} $\mathfrak{A}$ of the scheme.

From spectral theory, $\mathfrak{A}$ also has a basis of orthogonal
idempotents $E_0,\dots,E_d$ with $E_0+\dots+E_d = I$.
It is convenient to index $E_0=\frac{1}{N} J$, which is evidently one of the
idempotents.

The \emph{Hadamard} or \emph{entrywise} product of matrices $A,B \in \C^{N
\times N}$ is $A \circ B$ with $(A \circ B)(i,j) = A(i,j)B(i,j)$.  For the
Bose-Mesner algebra, we have the dual relations
$$A_i \circ A_j = \delta_{ij} A_i ~~~~ \text{and} ~~~~ E_i  E_j =
\delta_{ij} E_i.$$
In particular, since $\mathfrak{A}$ is closed under $\circ$, a parallel
version of (\ref{structure}) exists; namely,
\begin{equation*}
%\label{Krein}
E_i \circ E_j = \frac{1}{N} \sum_{k=0}^d q_{ij}^k E_k
\end{equation*}
for some positive reals $q_{ij}^k$.  These are called the \emph{Krein
parameters} for the scheme.

The basis change matrices between $(A_0,\dots,A_d)$ and $(E_0,\dots,E_d)$
are denoted by $P$ and $\frac{1}{N}Q$; explicitly, the entries are given by
$$A_i = \sum_{j=0}^d P(i,j) E_j ~~~~ \text{and} ~~~~ E_i  = \frac{1}{N}
\sum_{j=0}^d Q(i,j) A_j.$$
The $E_j$ being projections imply that the entries $P(i,j)$ are the
eigenvalues of the matrices $A_i$, with the column space of $E_j$ as the
associated eigenspace.  For this reason, the matrices $P,Q$ are called,
respectively, the \emph{first} and \emph{second eigenmatrix}.  They are of
course related by $Q=NP^{-1}$.

For a subset $Y \subseteq X$, its \emph{indicator vector} is $\vj_Y \in
\{0,1\}^X$, with $\vj_Y(x)=1$ if and only if $x \in Y$. The \emph{inner
distribution} of a nonempty $Y$ is $\mathbf{a} = (a_0,\dots,a_d)$, where
$$a_i = \frac{1}{|Y|} \vj_Y^\top A_i \vj_Y = \frac{1}{|Y|} |Y^2 \cap R_i|.$$
The following observation leads to the famous Delsarte LP bound for cliques
in association schemes.
\begin{thm}[(Delsarte, \cite{Del})]
\label{delsarte}
The inner distribution $\mathbf{a}$ of a nonempty subset of points in an
association scheme with second eigenmatrix $Q$ satisfies $\mathbf{a} Q \ge
\mathbf{0}$.
\end{thm}

Now consider a pointed set $(X,x)$ with $x \in X$.  Define diagonal
zero-one matrices $E_0',\dots,E_d'$ in $M_X(\C)$ by
$$E_i'(y,y) =
\begin{cases}
1 & \text{if}~(x,y) \in R_i, \\
0 & \text{otherwise}.
\end{cases}
$$
Observe that, by analogy with the Bose-Mesner idempotents, these are
projections with $E_0'$ of rank one and $E_0'+\dots+E_d'=I$.

The algebra $\mathfrak{T} = \langle A_i,E_i' \rangle$ obtained by extending
$\mathfrak{A}$ by these $E_i'$ is
called the \emph{subconstituent} or \emph{Terwilliger algebra} of the scheme
with respect to $x$.  Unlike $\mathfrak{A}$, the algebra $\mathfrak{T}$ is
not in general commutative.

Completing the duality, we have $A_0',\dots,A_d'$ with
$$A_i'(y,y) = N E_i(x,y).$$

After this introduction, we are now interested exclusively in the conjugacy
scheme on $\Sym(n)$.
Recall that the conjugacy classes (and irreducible characters) of $\Sym(n)$
are in correspondence with the integer partitions of $n$.  For a fixed $n$,
we will use $m$ to denote the number of partitions, and index the conjugacy
classes as $C_0,\dots,C_{m-1}$, with $C_0=\{\id\}$ consisting only of the
identity permutation.
For any fixed $\theta \in C_k$, we have
$$p_{ij}^k = \left| \{(\phi,\psi) \in C_i \times C_j : \phi \psi = \theta \}
\right|.$$
Alternatively, as stated in \cite{Jackson},
$$p_{ij}^k = \frac{|C_i||C_j|}{n!} \sum_{\chi}
\frac{\chi(\phi_i) \chi(\phi_j) \chi(\phi_k)}{\chi(\id)},$$
where the sum is over all $m$ irreducible characters, and where $\phi_i$ is
a representative in class $i$, etc.

The Bose-Mesner algebra $\mathfrak{A}$ is a commutative sub-algebra of
$M_{n!}(\C)$.  Its eigenvalues are obtained nearly directly from the
character table of $\Sym(n)$.

\begin{lemma}[(Tarnanen, \cite{T})]
\label{Qmatrix}
The second eigenvalue matrix $Q$ for the conjugacy scheme on $\Sym(n)$ is
given by
$$Q(i,j) = \chi_j(\id) \chi_j (\phi_i),$$
where $0 \le i,j < m$ index both the conjugacy classes and irreducible
characters.
\end{lemma}

The idempotents and Krein parameters of this conjugacy scheme are also
expressible via the irreducible characters, but we omit the details.

The linear programming bound of Theorem~\ref{delsarte} now becomes
$$\begin{array}{lll}
\mbox{maximize:} & a_0 + a_1 + \dots + a_{m-1}  \\
\mbox{subject to:} & \sum_{0 \le i < m} a_i \chi_k(\phi_i) \ge 0 &
\text{for~}0\le k<m,\\
&a_0=1, a_i \ge 0, & \text{and}\\
&a_i= 0  &\text{if~}d_H(\id, \phi_i) \not\in D.\\
\end{array}$$

Given $n$ and $D$, let $M_{LP}$ denote this maximum value.
Then
$$M(n,D) \le M_{LP}.$$
This was used with some success in \cite{B,T} to obtain some upper bounds
on permutation codes with $n < 15$.  Additionally, the LP
led to a general upper bound on $M(n,4)$ in \cite{DS}.

When extending to the Terwilliger algebra in this setting, it is natural to
take $\id \in \Sym(n)$ as the distinguished element.
We have
$$E_i'(\phi,\phi) =
\begin{cases}
1 & \text{if}~\phi \in C_i, \\
0 & \text{otherwise}.
\end{cases}
$$
The products $E_i' \mathfrak{A} E_j'$ are supported on the $C_i \times C_j$
block.  This offers a natural decomposition of the Terwilliger algebra in
general.  Note that $E_i' A_k E_j'$ is the zero matrix if and only if
$p_{ij}^k = 0$.

\Section{Isometries, orbits, and the centralizer algebra}

Consider the automorphism group $\Aut(\A)$ of the Bose-Mesner algebra of
$\Sym(n)$.  On one hand, this is the intersection of the automorphism groups
of the graphs whose adjacency matrices are the zero-one generators $A_i$, $0
\le i < m=p(n)$.  Alternatively, this is also the group of isometries of
$\Sym(n)$, when endowed with the
Hamming distance metric $d_H$.  Let us also denote this group by $\Iso(n)$.

Let $\cL$ and $\cR$ denote the group actions of left and right
multiplication by $\Sym(n)$, acting on itself.  Let $\cI$
be the (involutoric) group action generated by inversion $\phi \mapsto
\phi^{-1}$ on $\Sym(n)$.  That these actions induce all isometries appears
in \cite{F}, and is discussed further in \cite{B}.

\begin{lemma}
$\Aut(\A) = \Iso(n) = (\cL \times \cR) \rtimes \cI \cong \Sym(n) \wr 2.$
\end{lemma}

Let $\Iso_1(n)$ be the subgroup of isometries which fixes the identity
element.  Its action on $\Sym(n)$ is generated by conjugations $\cC$ and
inversion $\cI$.  These actions commute, and therefore $\Iso_1(n)=\cC \times
\cI \cong \Sym(n) \times 2$.

For the extension to the Terwilliger algebra $\mathfrak{T}$, we are
concerned with how $\Iso_1(n)$ acts on pairs of permutations.  In what
follows, sums indexed by `$\chi$' are to be taken over all irreducible
characters $\chi$ of $\Sym(n)$.

\begin{prop}
\label{orbit1}
The number of orbits of $\Sym(n)^2$, acted on by conjugation and inversion
is given by
$$b_n = \frac{1}{2n!} \sum_{\alpha \in \Sym(n)} \left[ \left( \sum_{\chi}
\chi^2(\alpha) \right)^2
+\left( \sum_{\chi} \chi(\alpha) \right)^3\right].$$
\end{prop}

\begin{proof}
We use Burnside's orbit counting lemma to obtain
\begin{eqnarray*}
b_n &=& \frac{1}{|\cC \times \cI|} \sum_{g \in \cC \times \cI}
|\mathrm{fix}(g)| \\
&=& \frac{1}{2n!} \left[ \sum_{\alpha \in \Sym(n)} |\{(\phi,\psi):\alpha
(\phi,\psi) \alpha^{-1} =  (\phi,\psi)\}| + \sum_{\alpha \in
\Sym(n)}|\{(\phi,\psi):\alpha (\phi,\psi) \alpha^{-1} =
(\phi^{-1},\psi^{-1})\}|\right] \\
&=& \frac{1}{2n!} \left[ \sum_{\alpha \in \Sym(n)} |\{\phi:\alpha \phi
\alpha^{-1} =  \phi\}|^2 +  \sum_{\alpha \in \Sym(n)} |\{\phi:\alpha \phi
\alpha^{-1} =  \phi^{-1}\}|^2 \right].
\end{eqnarray*}
Both terms in the sum are connected with interesting class functions on
$\Sym(n)$.  The first term is just the sum of squared centralizers.
Recall that orthonormality of the irreducible characters implies that for an
element $\alpha \in \Sym(n)$,
$$\sum_\chi \chi(\alpha)^2 = |C(\alpha)|,$$
the (size of the) centralizer of $\alpha$.
So the first of our two summation terms can be rewritten
\begin{equation}
\label{first-term}
\sum_{\alpha \in \Sym(n)} \left( \sum_{\chi} \chi^2(\alpha) \right)^2.
\end{equation}
For the second sum, observe that the condition $\alpha \phi \alpha^{-1} =
\phi^{-1}$ is equivalent to $(\alpha \phi)^2 = \alpha^2$.  So for a fixed
$\alpha$, the number of such $\phi$ is simply the number of square roots of
$\alpha^2$ in $\Sym(n)$.  As an aside, this counting problem for general
groups $G$ was connected long ago (see \cite{W}) with the `Frobenius-Schur
indicator'
$$s(\chi) = \frac{1}{|G|} \sum_{g \in G} \chi(g^2).$$
For our purposes, all irreducible representations of $\Sym(n)$ are real and
therefore $s(\chi)$ is always $1$.
The number of square roots of an element $\beta$ reduces in this case to the
basic character sum
$\sum_{\chi} \chi(\beta).$
For $\beta=\alpha^2$, the term we seek becomes
\begin{eqnarray}
\nonumber
\sum_{\alpha \in \Sym(n)} \left( \sum_{\chi} \chi(\alpha^2) \right)^2 &=&
\sum_{\beta \in \Sym(n)} \left( \sum_{\chi} \chi(\beta) \right) \left(
\sum_{\chi} \chi(\beta) \right)^2\\
\label{second-term}
& =& \sum_{\alpha \in \Sym(n)} \left( \sum_{\chi} \chi(\alpha) \right)^3.
\end{eqnarray}
Combining (\ref{first-term}) and (\ref{second-term}) completes the proof.
\end{proof}

Now consider $\Iso_1^*(n) \cong \Sym(n) \times 2 \times 2$, acting on
$\Sym(n)^2$ by conjugation, inversion, and the `coordinate swaps'
$(\phi,\psi) \mapsto (\psi,\phi)$.  A small extension of
Proposition~\ref{orbit1} also yields a character sum formula to count orbits
for this action.

\begin{prop}
The number of orbits of $\Sym(n)^2$, acted on by conjugation, inversion, and
coordinate swaps, is given by
$$b^*_n = \frac{1}{2} \left[ b_n + \frac{1}{n!}\sum_{\alpha \in \Sym(n)}
\left( \sum_{\chi} \chi(\alpha) \right)
\left( \sum_{\chi} \chi^2(\alpha) \right) \right],$$
where $b_n$ is as in Proposition~\ref{orbit1}.
\end{prop}

\begin{proof}
New sums not already present in the proof of Proposition 3.2 count, for a
fixed $\alpha$,
$$|\{(\phi,\psi):\alpha (\phi,\psi) \alpha^{-1} =
(\psi,\phi)\}|~~\text{and}~~ |\{(\phi,\psi):\alpha (\phi,\psi) \alpha^{-1} =
(\psi^{-1},\phi^{-1})\}|.$$
Since conjugacy by $\alpha$ is a bijection, it follows that each of these is
counted simply by the centralizer of $\alpha^2$.
Alternatively, we may count, over all $\beta \in \Sym(n)$, the number of
square roots of $\beta$ times the size of the centralizer of $\beta$.
Using the expressions in the previous proof,
$$\sum_{\alpha \in \Sym(n)} |C(\alpha^2)| = \sum_{\beta \in \Sym(n)}
\left( \sum_{\chi} \chi(\beta) \right) \left( \sum_{\chi} \chi^2(\beta)
\right).$$
This, along with earlier terms, finishes the orbit count.
\end{proof}

Table~\ref{orbits} contains the first few interesting values of $b_n$ and $b_n^*$.

\begin{table}[h]
\normalsize
$$\begin{array}{c|r|r|r}
\hline
n & n! & b_n & b_n^* \\
\hline
4 & 24 & 43 & 28 \\
5 & 120 & 155 & 93 \\
6 & 720 & 761 & 425 \\
7 & 5040 & 4043 & 2151 \\
8 & 40320 & 27190 & 14016 \\
\hline
\end{array}$$
\caption{Counting orbits of $\Sym(n)^2$ under $\Iso_1(n)$ and $\Iso_1^*(n)$.}
\label{orbits}
\end{table}

Let $O_1,\dots,O_M$ denote the orbits of $\Sym(n)^2$ under $\Iso_1(n) \cong
\Sym(n) \times 2$.  Define the corresponding zero-one matrices
$B_1,\dots,B_M \in M_{n!}(\C)$ by
$$B_l(\phi,\psi) = \begin{cases}
1 & \text{if}~(\phi,\psi) \in O_l,\\
0 & \text{otherwise}.
\end{cases}$$
Put $\mathfrak{B} = \langle B_l : l = 1,\dots,M \rangle$ and observe that by
Proposition~\ref{orbit1} we have $\dim(\mathfrak{B})=b_n$.

Alternatively, $\mathfrak{B} = {\rm End}_G(X)$, where $X=\Sym(n)$ and
$G=\Iso_1(n)$.  For $\alpha \in \Sym(n)$, consider
the zero-one matrix $\widehat{\alpha} \in M_{n!}(\C)$ with
$$\widehat{\alpha}(\phi,\psi) = \begin{cases}
1 & \text{if}~\alpha^{-1} \phi \alpha = \psi,\\
0 & \text{otherwise}.
\end{cases}$$
Then $\alpha \mapsto \widehat{\alpha}$ is the conjugacy representation.
Define the action of inversion similarly
$$\widehat{\iota}(\phi,\psi) = \begin{cases}
1 & \text{if}~\phi^{-1} = \psi,\\
0 & \text{otherwise}.
\end{cases}$$
Let
$\mathfrak{K} = \langle \widehat{\iota},\widehat{\alpha} : \alpha \in
\Sym(n) \rangle$.  The following rephrases the earlier definition of orbits
under $\Iso_1(n)$ in terms of matrix algebras.

\begin{prop}
$\mathfrak{B}$ is the centralizer algebra of $\mathfrak{K}$.
\end{prop}

Since conjugacy classes of $\Sym(n)$ are left invariant by both conjugation
and inversion, the orbits under $\Iso_1(n)$ admit a classification according
to conjugacy classes.  In other words, each matrix $B_l$ is supported on
some block $C_i \times C_j$.  The sum of all such $B_l$ supported on a given
block is, when restricted to that block, the all-ones matrix $J$.

With notation from Section 2, we recall that the `restricted' Bose-Mesner
algebra $E_i' \mathfrak{A} E_j'$ has conjugation-invariant generators
supported on the $C_i \times C_j$ block.  So it follows that $\mathfrak{T}$
is a sub-$\C$-algebra of $\mathfrak{B}$.  The reverse inclusion seems
reasonable, but much more difficult to prove.  Indeed, the dimensions in
Table~\ref{orbits} grow much faster than the number of generators for
$\mathfrak{T}$.

\begin{conjecture}
\label{eq-algebras}
The Terwilliger algebra of the conjugacy scheme of $\Sym(n)$ has as a
basis the zero-one matrices $B_i$, $i=1,\dots,M$, defined via orbits of
$\Sym(n)^2$ under $\Iso_1(n)$.  That is, as $\C$-algebras, we have
$\mathfrak{T}=\mathfrak{B}$.
\end{conjecture}

Conjecture~\ref{eq-algebras} was verified up to $n = 5$ in \cite{BO}.  We
have extended this verification to $n=6$ but leave the general problem to
future work.

Fortunately, the center of $\mathfrak{B}$ is reasonably easy to describe.  The following appears in \cite{NS} for general groups; here (and from now on) we are interested in $G=\Iso_1(n) \cong \Sym(n) \times 2$.

\begin{prop}
The primitive central idempotents of $\mathfrak{B}$ are given by
\begin{equation}
\label{epsilon}
\epsilon_k = \frac{\chi_k(\id)}{2n!} \sum_{\gamma \in G} \chi_k(\gamma)
\widehat{\gamma},
\end{equation}
where $k$ indexes the irreducible characters of $G$.
\end{prop}

Since the $\epsilon_k$ commute, they have a common basis of eigenvectors.
Consider a unitary matrix $U \in M_{n!}(\C)$ which has orthonormal
eigenvectors of the $\epsilon_k$, organized by columns.  Then $U
\mathfrak{B} U^\top$ is block-diagonalized by simple blocks.  Our goal
later, in Section 5, is to decompose these simple blocks into basic blocks.
For now, though, we have sufficient background to introduce
semidefinite programming for permutation codes.

\Section{Semidefinite programming}

Let $\Gamma$ be a subset of $\Sym(n)$, and recall that $\vj_\Gamma$ is the
$n!\times 1$ zero-one indicator vector of $\Gamma$.

%$$(\vj_\Gamma)_j=\left\{ \begin{array}{l  l} 1  & \mbox{ if } \phi_j \in\Gamma\\ 0 & \mbox {  otherwise.}\end{array} \right. $$

Define now the subgroup $\Pi$ of $\Iso(n)$ consisting
of all isometries $t$ such that $\id \in t(\Gamma)$.  Let $\Pi'$ denote the
complement of this subgroup, consisting of all $t' \in \Iso(n)$ with $\id
\not\in t'(\Gamma)$.  It is clear that $|\Pi|=2n!|\Gamma|$, while
$|\Pi'|=2n!(n!-|\Gamma|)$.
This defines two matrices
 $$R_\Gamma =\frac{1}{|\Pi|} \sum_{t \in \Pi} \vj_{t(\Gamma)} \cdot
\vj^\top_{t(\Gamma)}~~\text{and}~~
R'_\Gamma =\frac{1}{|\Pi'|} \sum_{t' \in \Pi'} \vj_{t'(\Gamma)} \cdot
\vj^\top_{t'(\Gamma)}$$
which are roughly analogous to the inner distribution vector $\mathbf{a}$ of
Section 2.

By construction, both $R_\Gamma$ and $R'_\Gamma$ are
positive semidefinite (symmetric) matrices.  And the trace of $R_\Gamma$
gives the cardinality of $\Gamma$. The $(\id,\id)$-entry of $R_\Gamma$ is
always $1$, since $\id$ appears in all copies under $\Pi$. The the
$(\id,\phi)$-entry equals the $(\phi,\phi)$-entry in $R_\Gamma$, and,
$R_\Gamma$ being symmetric, it is also equal to the $(\phi, \id)$-entry.
Since $\Iso(n)$ is transitive on $\Sym(n)$,
one has the following formula and lemma.
\begin{equation}
\label{Rmat}
R_\Gamma(\phi,\psi)+\frac{n!-|\Gamma|}{|\Gamma|}R_\Gamma' (\phi,\psi)=
R_\Gamma(\phi\psi^{-1},\phi\psi^{-1})=R_\Gamma(\id,\phi\psi^{-1}).
\end{equation}
(In particular, the matrix $R_\Gamma+\frac{n!-|\Gamma|}{ |\Gamma|}R_\Gamma'$
has a diagonal of $1$.)
\begin{lemma}
Consider an entry $(\phi,\psi) \in \Sym(n)^2$.
If no quotient in $\Gamma$ belongs to the same conjugacy class as $\phi
\psi^{-1}$, then $R_\Gamma(\phi,\psi) = R'_\Gamma(\phi,\psi) = 0$.
\end{lemma}
Next, we have the key connection with the centralizer algebra $\mathfrak{B}$ discussed in
Section 3.
\begin{prop}
The two matrices $R_\Gamma$ and $R'_\Gamma$ belong to the algebra
$\mathfrak{B}$.
\end{prop}

\begin{proof}
Recall that $\{B_l\}$ are the zero-one matrices representing the orbits
$\{O_l\}$ of $\Sym(n)^2$ under the action of $\Iso_1(n)$.  With the
convention that $O_1=\{(\id,\id)\}$, one has $R_\Gamma = B_1+\sum_l a_l
B_l$,
where the coefficients $a_l$ are explicitly determined by the formula
$$a_l=\frac{|\{ (\theta, \mu, \nu)\in \Gamma \times  \Gamma \times \Gamma :
(\theta^{-1}\mu, \theta^{-1}\nu) \in O_l  \}|}{|\Gamma|\ |O_l|}.$$
Observe that $0 \le a_l \le 1$.

Let $I_1$ denote the subset of indices $l\neq 1$ such that $O_l$ contains
elements of type $(\id,\phi)$, $ (\phi, \id)$ or $(\phi, \phi)$, and let
$I_2$ be the complementary set $\{2\dots M\} \backslash I_1$.

With some rearranging and (\ref{Rmat}), it follows easily that
$$\frac{n!- |\Gamma|}{|\Gamma|}R'_\Gamma = (I-B_1)+ \sum_l a_l B'_l,$$ where
the matrices $B'_l$ are defined by $B'_l=-B_l$
if $l \in I_2$ and by the following expression if $l \in I_1$:
$$B'_l(\phi,\psi)=\left\{\begin{array}{l l} 1 & \mbox{ if } (\phi\psi^{-1},
\id)\in O_l,\\
-1 & \mbox{ if } \phi=\psi\mbox{ and }(\phi , \id)\in O_l,\\
0 & \mbox{ otherwise.}\end{array}\right.$$
In both cases, we have the required matrices expressed as linear
combinations of the $\{B_l\}$.
\end{proof}

We now summarize the above facts concerning $R_\Gamma$ and $R'_\Gamma$.

\begin{prop}
Let $D \subseteq [0,n]$.  For any $(n,D)$-permutation code $\Gamma \subseteq
\Sym(n)$, we have
$$\left\{ \begin{array}{l l}
|\Gamma|= \trace(R_\Gamma) & \\
R_\Gamma=B_1+\sum_i a_i B_i  \succeq 0\\
R'_\Gamma =I-B_1+\sum_j a_j B'_j\succeq 0\\
R_\Gamma(\id, \id)=1& \\
R'_\Gamma(\id, \id)=0&\\
R_\Gamma(\phi, \psi)=R'_\Gamma(\phi, \psi) = 0 & \text{if~}d_H(\phi, \psi)
\not\in D\\
%{R'_\Gamma}(\phi, \psi)=0 &  \text{if~}\phi \psi^{-1} \in C_i \text{~for~} i \not\in D
\end{array}\right.$$
\end{prop}

From this, the following SDP system is obtained, considering the
coefficients $a_i$ as variables in $[0,1]$:

\begin{equation}
\label{sdp-statement}
\begin{array}{lll}
\mbox{maximize:} & \trace(R_1)  \\
\mbox{subject to:} & R_1=B_1+\sum_i a_i B_i  \succeq 0\\
&R_2 =I-B_1+\sum_i a_i B'_i\succeq 0, & 0\le  a_i \le 1\\
&R_1(\id, \id)=1,& \\
&R_2(\id, \id)=0,&\\
&R_1(\phi, \psi)=R_2(\phi, \psi) = 0 & \text{if~}d_H(\phi, \psi) \not\in
D.\\
\end{array}
\end{equation}

Given $n$ and $D$, let $M_{SDP}$ denote the maximum value of this program.
Then
$$M(n,D) \le M_{SDP}.$$

Recall that the matrices $B_i$ have size $n!\times n!$.  So presently, we
can only directly consider (\ref{sdp-statement}) for $n\le 6$. Since permutation
codes are well understood for $n \le 5$, our main contribution at this stage
is a table of SDP bounds for $n=6$.
For $n \ge 7$, it is necessary to consider an equivalent SDP obtained via
block-diagonalization of the algebra $\mathfrak{B}$.  That is the topic of
the next section.

We conclude with Table~\ref{sdp6}, indicating the SDP bounds from (\ref{sdp-statement}) for $n=6$ and
various distance sets $D$. For comparison, the Delsarte LP bounds $M_{LP}$ are given, along with the quantity $\Pi_{d\in D} d$, which is 
an upper bound either for $M(n,D)$ or $M(n,D^c)$ (see \cite{T}).

\begin{table}[h]
\normalsize
$$\begin{array}{|l| c |c | c|}
\hline
D& M_{SDP} & M_{LP} & \Pi_{d \in D} d\\
\hline
\emptyset & 1& 1  & 1\\
\hline
\{ 6\} & 6& 6  & 6\\
\hline
\{ 5\} & 12& 15  & 5\\
\hline
\{ 4\} & 7& 12  & 4\\
\hline
\{ 3\} & 5& 6  & 3\\
\hline
\{ 2\} & 2& 2  & 2\\
\hline
\{ 5,6\} & 25& 30  & 30\\
\hline
\{ 4,6\} & 35& 48  & 24\\
\hline
\{ 3,6\} & 18& 24  & 18\\
\hline
\{ 2,6\} & 9& 12  & 12\\
\hline
\{ 4,5\} & 20& 20  & 20\\
\hline
\{ 3,5\} & 14& 15  & 15\\
\hline
\{ 2,5\} & 12& 15  & 10\\
\hline
\{ 3,4\} & 15& 24  & 12\\
\hline
\{ 2 ,4\} & 8& 12  & 8\\
\hline
\{ 2,3\} & 6& 6  & 6\\
\hline
\end{array}
~~
\begin{array}{|l| c |c | c|}
\hline
D& M_{SDP} & M_{LP} & \Pi_{d \in D} d\\
\hline
\{ 4,5,6\} & 120& 120  & 120\\
\hline
\{ 3,5,6\} & 56 & 60  & 90\\
\hline
\{ 2,5,6\} & 27& 30  & 60\\
\hline
\{ 3,4,6\} & 39& 48  & 72\\
\hline
\{ 2,4,6\} & 48& 48  & 48\\
\hline
\{ 2,3,6\} & 20& 36  & 36\\
\hline

\{ 3,4,5\} & 60& 60  & 60\\
\hline
\{ 2,4,5\} & 24& 30  & 80\\
\hline
\{ 2,3,5\} & 14& 15  & 30\\
\hline
\{ 2, 3 ,4\} & 24& 24  & 24\\
\hline

\{ 3,4,5,6\} & 360& 360  & 360\\
\hline
\{ 2,4,5,6\} & 120 & 120  & 240\\
\hline
\{ 2,3,5,6\} & 56& 60  &180\\
\hline
\{ 2,3,4,6\} & 48& 48  & 144\\
\hline
\{ 2,3,4,5\} & 120& 120  & 120\\
\hline
\{ 2,3,4,5,6\} & 720& 720  & 720\\
\hline
\end{array}$$
\caption{LP and SDP bounds for $\Sym(6)$}
\label{sdp6}
\end{table}

\Section{Block diagonalization and bounds for $\Sym(7)$}

First, let us recall the notion of  $*$-isomorphism. A bounded linear map
$\kappa :
\mathcal{A} \mapsto \mathcal{B}$ such that for all $X,Y \in
\mathcal{A}$, $\kappa(XY) =\kappa(X)\kappa(Y)$ and
$\kappa(X^*)=\kappa(X)^*$ is called a $*$-\emph{homomorphism} of
algebras. Such a map which is also bijective is naturally called a $\C^*$-\emph{isomorphism} and the two
algebras $\mathcal{A}, \mathcal{B}$ are $*$-\emph{isomorphic}. A $*$-isomorphism
of matrix algebras preserves positive semidefiniteness. 

Next, we follow Gijswijt in \cite{Gij} and Vallentin in \cite{V} by considering the $*$-isomorphism which `block-diagonalizes' our algebra $\mathfrak{B}$. 
%There are integers $d$, and $m_1 ,\dots, m_d$ so that there is an algebra isomorphism 
%$\kappa: \mathfrak{B}\rightarrow \bigoplus_{k=1}^d \C^{m_k\times m_k}$.

In more detail, we can regard $\mathfrak{K}$ as a semisimple $G$-module, $G=\Iso_1(n)$, acted upon
by $\mathfrak{B}$.  Therefore, $\mathfrak{K}$ decomposes as a sum of orthogonal
submodules $W_k$, where $W_k$ has columns forming a basis of the range
of the matrix $\epsilon_k$ defined in (\ref{epsilon}), for $0\le k < 2m$.
Each $W_k$ is a direct sum of, say, $m_k$ submodules: $W_k=V_k ^1\oplus \dots \oplus V_k ^{m_k}$, with $V_k^i\cong V_k ^j$ for all $i,j$. Let $d_k$ denote the common dimension of $V_k ^i$.  Incidentally, this parameter is also given by $d_k=\chi_k (\id)$ where $\chi_k$ is the $k$-th irreducible character of $G$.

For the centralizer algebra $\mathfrak{B}$ of $\mathfrak{K}$, we correspondingly have a decomposition of $\mathfrak{B}$ as an orthogonal direct sum of, say, $W'_k$, each of which is in turn decomposes into $d_k$ submodules of dimension $m_k$.

The action of elements in $\mathfrak{K}$ can be naturally decomposed along each orbit (i.e.~conjugacy class). Let $\chi\vert_{C_l}$ denote the restriction of the character $\chi$ on the conjugacy class $C_l$.  So $\chi\vert_{C_l} (\gamma)$ is the number of permutations $ \alpha \in C_l$ fixed by $\gamma \in \Iso_1 (n)$. Depending on the conjugacy class of $\gamma$, it can be written $\gamma (\alpha)= \beta \alpha \beta^{-1}$ or $\gamma (\alpha)= \beta \alpha^{-1} \beta^{-1}$ for all $\alpha\in \Sym(n)$. In the first case, $\chi\vert_{C_l} (\gamma) $ is the number of permutations in $C_l$ commuting with a fixed permutation $\beta\in C_j$. In the latter case, $\chi\vert_{C_l} (\gamma) $ is the number of permutations $\alpha$  in $C_l$ such that, for a fixed $\beta \in C_j$, $(\alpha \beta)^2= \beta^2$.  So far, to the best of our knowledge, no closed form formula is known for these two expressions.

For each $l$, $\chi\vert_{C_l}$ splits as a sum of irreducible characters $\chi_k$, $0\le k < 2m$, of $\Iso_1(n)$ as
$$\chi\vert_{C_l} = \sum_{k=0}^{2m-1} a_k^l \chi_k.$$ 
The coefficients are given by 
$$a_k^l=\langle \chi\vert_{C_l} , \chi_k \rangle =\frac{1}{2 n!} \sum_{\gamma \in G} \chi\vert_{C_l} (\gamma)  \chi_k (\gamma).$$ 
So each irreducible representation of $G=\Iso_1(n)$ appears $m_k = \sum_{l} a_k^l $ times  in the action of $\mathfrak{K}$ and the dimension $b_n$ of $\mathfrak{B}$ is equal to $\sum_k m_k^2$. The non-zero values of $m_k$ for $4 \le n \le 7$ are given in Table~\ref{Mj}.  

\begin{table}[h]
\normalsize
$$\begin{array}{|l|l|}
\hline
n& m_k \\ 
\hline
4 & 1, 2, 2, 3, 5\\
\hline
5 & 1, 1,3,3, 5, 5, 6,7\\
\hline
6&   1, 1,  1, 1, 1, 1, 3, 3, 4,  6, 7, 8, 8, 9, 9, 11, 15\\
\hline
7 & 1, 2, 2, 2, 2, 3, 3, 4, 5, 5, 7, 8, 9, 9, 13, 15, 15, 15, 16, 17, 19, 20, 20, 21, 26 \\ \hline
\end{array}$$
\caption{Values of non zero $m_k$ for small values of $n$.}
\label{Mj}
\end{table}

It remains to explicitly decompose each submodule $W'_k$ of $\mathfrak{B}$ into smaller `basic blocks'.

The method begins as follows. For each $k$, compute an element $v_k\in W_k$. Then define $U_k := \mathfrak{B} v_k$, a subspace of $W_k$.
Note that the vector $v_k$ can be obtained as $\epsilon_k x$ for any vector $x$ since $\epsilon_k $ is
idempotent. So, for instance, any non-zero row of $\epsilon_k$ affords a choice of $v_k$.

By Schur's lemma, the subspace $U_k$ meets each $V_k^i$ in a $1$-dimensional
subspace.  Then any $B_i \in \mathfrak{B}$ acts on $U_k$ by multiplication.  
If $(u_1, \dots, u_{m_k})$ is an orthogonal basis of $U_k$, then the basic block $B_{ik}:=B_i\vert_{U_k}$ of
$B_i$ can be computed as
\begin{eqnarray}
\label{basic-formula}
B_{ik}(j,l)= \frac{\langle B_i u_j, u_l \rangle}{||u_j||\ ||u_l||}
\end{eqnarray}
Each basic block $B_{ik}$ is an $m_k\times m_k$ matrix. Recall the matrix $B_i$ has
nonzero entry in the $(\phi, \psi)$ position iff $(\phi, \psi)\in O_i$. So the inner product in the numerator of (\ref{basic-formula}) is
$$\langle B_iu_j, u_l \rangle= \sum_{(\phi, \psi) \in O_i} u_j(\phi) u_l(\psi).$$

\begin{table}[h]
\normalsize
$$\begin{array}{|l| c |c | c|}
\hline
D& M_{SDP} & M_{LP} & \Pi_{d \in D} d\\
\hline
\emptyset & 1& 1  & 1\\
\hline
\{ 7\} & 7& 7  & 7\\
\hline
\{ 6\} & 28& 30  & 6\\
\hline
\{ 5\} & 14& 15  & 5\\
\hline
\{ 4\} & 9& 12  & 4\\
\hline
\{ 3\} & 6& 8  & 3\\
\hline
\{ 2\} & 2& 2  & 2\\
\hline
\{ 6,7\} & 42& 42  & 42\\
\hline
\{ 5,7\} & 47& 52  & 35\\
\hline
\{ 4,7\} & 34& 46  & 28\\
\hline
\{ 3,7\} & 27& 42  & 21\\
\hline
\{ 2,7\} & 12& 14  & 14\\
\hline
\{ 5,6\} & 30& 30  & 30\\
\hline
\{ 4,6\} & 47& 72  & 24\\
\hline
\{ 3,6\} & 31& 36  & 18\\
\hline
\{ 2,6\} & 34& 48  & 12\\
\hline
\end{array}
~~
\begin{array}{|l| c |c | c|}
\hline
D& M_{SDP} & M_{LP} & \Pi_{d \in D} d\\
\hline
\{ 4,5\} & 26& 60  & 20\\
\hline
\{ 3,5\} & 15& 15  & 15\\
\hline
\{ 2,5\} & 14& 15  & 15\\
\hline
\{ 3, 4\} & 21& 26  & 12\\
\hline
\{ 2,4\} & 9& 12  & 8\\
\hline
\{ 2,3\} & 7& 9  & 6\\
\hline
\{ 5,6,7\} & 134& 140  & 210\\
\hline
\{4, 6,7\} & 181 & 205  & 168\\
\hline
\{3, 6,7\} & 70& 84  & 126\\
\hline
\{2, 6,7\} & 59& 84  & 84\\
\hline
\{4, 5,7\} & 74& 93  & 140\\
\hline
\{3, 5,7\} & 54& 63  & 105\\
\hline
\{2, 5,7\} & 47& 52  & 70\\
\hline
\{3,4, 7\} & 49& 84  & 84\\
\hline
\{2,4, 7\} & 36& 46  & 56\\
\hline
\{2,3, 7\} & 32& 42  & 42\\
\hline
\end{array}$$

$$\begin{array}{|l| c |c | c|}
\hline
D& M_{SDP} & M_{LP} & \Pi_{d \in D} d\\
\hline
\{4,5,6\} & 120& 120  & 120\\
\hline
\{3,5,6\} & 66& 108  & 90\\
\hline
\{2,5,6\} & 42& 60  & 60\\
\hline
\{3,4,6\} & 50& 72  & 72\\
\hline
\{2,4,6\} & 48& 72  & 48\\
\hline
\{2,3,6\} & 35 & 54  & 36\\
\hline
\{3,4,5\} & 60& 60  & 60\\
\hline
\{2,4,5\} & 32& 60  & 40\\
\hline
\{2,3,5\} & 15 & 15  & 30\\
\hline
\{2,3,4\} & 24 & 32  & 24\\
\hline
\{4,5,6,7\} & 535& 543  & 840\\
\hline
\{3,5,6,7\} & 380& 420  & 630\\
\hline
\{2,5,6,7\} & 160& 172  & 420\\
\hline
\{3,4,6,7\} & 207& 280  & 504\\
\hline
\{2,4,6,7\} & 181& 205  & 336\\
\hline
\{2,3,6,7\} & 74 & 84  & 252\\
\hline
\end{array}
~~
\begin{array}{|l| c |c | c|}
\hline
D& M_{SDP} & M_{LP} & \Pi_{d \in D} d\\
\hline
\{3,4,5,7\} & 81& 93  & 420\\
\hline
\{2,4,5,7\} & 90& 140  & 280\\
\hline
\{2,3,5,7\} & 54 & 63  & 210\\
\hline
\{2,3,4,7\} & 77 & 168  & 168\\
\hline
\{3,4,5,6\} & 360& 360  & 360\\
\hline
\{2,4,5,6\} & 120& 120  & 240\\
\hline
\{2,3,5,6\} & 77 & 108  & 180\\
\hline
\{2,3,4,6\} & 51 & 72  & 144\\
\hline
\{2,3,4,5\} & 120 & 120  & 120\\
\hline
\{3,4,5,6,7\} & 2520& 2520  & 2520\\
\hline
\{2,4,5,6,7\} & 612& 630  & 1680\\
\hline
\{2,3,5,6,7\} & 380 & 420  & 1260\\
\hline
\{2,3,4,6,7\} & 207 & 280  & 1008\\
\hline
\{2,3,4,5,7\}  & 120 & 168  & 840\\
\hline
\{2,3,4,5,6\} & 720 & 720  & 720\\
\hline
\{2,3,4,5,6,7\} & 5040 & 5040  & 5040\\
\hline
\end{array}$$
\caption{LP and SDP bounds for $\Sym(7)$}
\label{sdp7}
\end{table}

To summarize, we have a block-diagonalization of the centralizer algebra $\mathfrak{B}$.  

\begin{prop}
The algebra $\mathfrak{B}$ is $*$-isomorphic to $\bar{\mathfrak{B}}=\bigoplus_{k=0}^{m-1} \C^{m_k\times m_k}$
generated by the diagonal joins of basic blocks $B_{ik}$. 
\end{prop}

The $*$-isomorphism is, in principle, possible to compute via (\ref{basic-formula}) and some calculations involving the conjugacy representation of $\Sym(n)$.  In such a calculation, say with $n=7$, a challenging first step is the enumeration of the orbits of $\Sym(n)^2$ under $\Iso_1(n)$.  But in this calculation, when possible we may compute the corresponding entries of the basic blocks `on the fly' one orbit at a time without storing the (large) algebra $\mathcal{B}$ in memory.

In any case, this $*$-isomorphism induces a reduction of our semidefinite program along the lines of \cite{Gij,V}.

%Let us define the linear mapping
%$$\Phi: B_i \mapsto D_i:=
%\left[\begin{array} {c c  c c}
%B_{i1} & 0 & \hdots  &0\\
%0 & B_{i2} & \hdots  &0\\
%\vdots &\vdots & \ddots  &0\\
%0 & 0 & \hdots  & B_{im}\\
%\end{array}
%\right].$$

\begin{corollary}
The semidefinite program in $(\ref{sdp-statement})$ is invariant under block-diagonalization of each $B_i$, the resulting program giving an upper bound for $M(n,D)$.
\end{corollary}

After similar computations as in the previous section, we report SDP bounds for $n=7$ according to Table~\ref{sdp7}.

\Section{Concluding remarks}

For $n\ge 8$, the number of variables is presently too large for practical
implementation of the SDP, even after block-diagonalization.  But we are
hopeful that $n=8$ can soon be attacked with some new ideas for generating
orbits, block-diagonalizing, and possibly constraining the number of
variables.  In fact, should the SDP eventually scale up to $n=10$, there is
a chance that $M(10,9)<90$ can be proved in this way.  Following \cite{CKL},
this represents an opportunity for a possible alternate proof of the
nonexistence of projective planes of order 10.  This is probably a long way off, but we hope it offers some motivation for advancing the techniques in this paper.

\end{document}